\documentclass[11pt,a4paper]{amsart}
\usepackage{amsthm,amsmath,amssymb}
\usepackage[english]{babel}
\usepackage[normalem]{ulem}
\usepackage[utf8]{inputenc}
\usepackage{graphicx}
\usepackage[left=3.5cm, right=2.5cm,top=3.5cm,bottom=3.5cm]{geometry}
\usepackage{tikz}
\usepackage{pdfpages}
\usepackage{amsmath}
\usepackage{amsfonts}
\usepackage{mathrsfs} 
\usepackage{nicefrac} 
\usepackage{amsthm}
\usepackage{mathtools}
\usepackage[foot]{amsaddr}

\title{Analysis of a model for the dynamics of microswimmer suspensions} 
\author{Etienne Emmrich}
\address{Technische Universit\"{a}t Berlin, Institut f\"{u}r Mathematik\\
Stra{\ss}e des 17. Juni 136, 10623 Berlin, Germany}
\email{emmricht@math.tu-berlin.de}

\author{Lukas Geuter}
\email{geuter@math.tu-berlin.de}
\date{\today}

\begin{document}

\newcommand{\de}{\,\text{d}}
\renewcommand{\ll}[1]{\langle\hspace{-0.75mm}\langle{#1}\rangle\hspace{-0.75mm}\rangle}

\newcommand{\lnnn}{\left|\!\left|\!\left|}
\newcommand{\rnnn}{\right|\!\right|\!\right|}
\newcommand{\lj}{\left[\!\left[}
\newcommand{\rj}{\right]\!\right]}

\newcommand*{\uskw}{(\nabla\boldsymbol{u})_{\text{skw}}}
\newcommand*{\usym}{(\nabla\boldsymbol{u})_{\text{sym}}}
\newcommand*{\ueskw}{(\nabla\boldsymbol{u}_\varepsilon)_{\text{skw}}}
\newcommand*{\uesym}{(\nabla\boldsymbol{u}_\varepsilon)_{\text{sym}}}
\newcommand*{\uenskw}{(\nabla\boldsymbol{u}_\varepsilon^n)_{\text{skw}}}
\newcommand*{\uensym}{(\nabla\boldsymbol{u}_\varepsilon^n)_{\text{sym}}}
\newcommand*{\unskw}{(\nabla\boldsymbol{u}_n)_{\text{skw}}}
\newcommand*{\unsym}{(\nabla\boldsymbol{u}_n)_{\text{sym}}}
\newcommand*{\utskw}{(\nabla\boldsymbol{\tilde{u}})_{\text{skw}}}
\newcommand*{\utsym}{(\nabla\boldsymbol{\tilde{u}})_{\text{sym}}}
\newcommand*{\uskwt}{(\nabla\boldsymbol{u}(t))_{\text{skw}}}
\newcommand*{\usymt}{(\nabla\boldsymbol{u}(t))_{\text{sym}}}
\newcommand*{\ueskwt}{(\nabla\boldsymbol{u}_\varepsilon(t))_{\text{skw}}}
\newcommand*{\uesymt}{(\nabla\boldsymbol{u}_\varepsilon(t))_{\text{sym}}}
\newcommand*{\utskwt}{(\nabla\boldsymbol{\tilde{u}}(t))_{\text{skw}}}
\newcommand*{\utsymt}{(\nabla\boldsymbol{\tilde{u}}(t))_{\text{sym}}}
\newcommand*{\phib}{\boldsymbol{\varphi}}
\newcommand*{\psib}{\boldsymbol{\psi}}
\newcommand*{\zb}{\boldsymbol{z}}
\newcommand*{\Pb}{\boldsymbol{p}}
\newcommand*{\pb}{\boldsymbol{p}}
\newcommand*{\ub}{\boldsymbol{u}}
\newcommand*{\Pbe}{\boldsymbol{p}_\varepsilon}
\newcommand*{\pbe}{\boldsymbol{p}_\varepsilon}
\newcommand*{\ube}{\boldsymbol{u}_\varepsilon}
\newcommand*{\Pbez}{\boldsymbol{p}_{0,\varepsilon}}
\newcommand*{\pbez}{\boldsymbol{p}_{0,\varepsilon}}
\newcommand*{\ubez}{\boldsymbol{u}_{0,\varepsilon}}
\newcommand*{\Pbn}{\boldsymbol{p}_n}
\newcommand*{\pbn}{\boldsymbol{p}_n}
\newcommand*{\ubn}{\boldsymbol{u}_n}
\newcommand*{\Pbt}{\boldsymbol{\tilde{p}}}
\newcommand*{\pbt}{\boldsymbol{\tilde{p}}}
\newcommand*{\ubt}{\boldsymbol{\tilde{u}}}
\newcommand*{\Pbz}{\boldsymbol{p}_{0,\varepsilon}}
\newcommand*{\Pbnn}{\boldsymbol{p}_{0,\varepsilon}^n}
\newcommand*{\pbz}{\boldsymbol{p}_{0,\varepsilon}}
\newcommand*{\pbnn}{\boldsymbol{p}_{0,\varepsilon}^n}
\newcommand*{\vb}{\boldsymbol{v}}
\newcommand*{\vbt}{\tilde{\boldsymbol{v}}}
\newcommand*{\vbd}{\boldsymbol{v}_\delta}
\newcommand*{\db}{\boldsymbol{d}}
\newcommand*{\dbt}{\tilde{\boldsymbol{d}}} 
\newcommand*{\dbd}{\boldsymbol{d}_\delta}
\newcommand*{\qb}{\boldsymbol{q}}
\newcommand*{\qbt}{\tilde{\boldsymbol{q}}} 
\newcommand*{\qbtd}{\tilde{\boldsymbol{q}_\delta}}
\newcommand*{\qbd}{\boldsymbol{q}_\delta}
\newcommand*{\intd}{\,\mathrm{d}}
\newcommand*{\blambda}{\boldsymbol{\Lambda}}
\newcommand*{\btheta}{\boldsymbol{\beta}}
\newcommand*{\bdelta}{\boldsymbol{\delta}}
\newcommand*{\by}{\boldsymbol{y}}
\newcommand*{\fb}{\boldsymbol{f}}
\newcommand*{\bw}{\boldsymbol{w}}
\newcommand*{\bv}{\boldsymbol{v}}
\newcommand*{\bM}{\boldsymbol{M}}
\newcommand*{\bW}{\boldsymbol{W}}
\newcommand*{\bL}{\boldsymbol{L}}
\newcommand*{\bH}{\boldsymbol{H}}
\newcommand*{\bHp}{\boldsymbol{H}_\text{per}}
\newcommand*{\bHps}{\boldsymbol{H}_{\text{per},\sigma}}
\newcommand*{\Hps}{\mathbb{H}_{\sigma}}
\newcommand*{\Wps}{\mathbb{W}_{\sigma}}
\newcommand*{\Lps}{\mathbb{L}_{\sigma}}
\newcommand*{\bLs}{\boldsymbol{L}_\sigma}

\newcommand*{\Hp}{\mathbb{H}}
\newcommand*{\Wp}{\mathbb{W}}
\newcommand*{\Lp}{\mathbb{L}}
\newcommand*{\dt}{\,\mathrm{d}t}
\newcommand*{\intT}{\int_{0}^{T}\!}
\newcommand*{\DA}{\bH^2\cap\bH_{0,\sigma}^1}
\newcommand*{\DAO}{\bH^2(\Omega)\cap\bH_{0,\sigma}^1(\Omega)}

\newtheorem{definition}{Definition}
\newtheorem{theorem}{Theorem}
\newtheorem{lemma}{Lemma}
\newtheorem{corollary}{Corollary}
\theoremstyle{remark}
\newtheorem{remark}{Remark}

\maketitle

\begin{abstract}
In this paper, a model that was recently derived in \textsc{Reinken} et al.~\cite{reinken} to describe the dynamics of microswimmer suspensions is studied. In particular, the global existence of weak solutions, their weak-strong uniqueness and a connection to a different model that was proposed in \textsc{Wensink} et al.~\cite{wensink} is shown.
\end{abstract}

\section{Introduction}

Microswimmers are small, self-propelling particles, e.g.~bacteria like \emph{Bacillus subtilis}, algae like \emph{Chlamydomonas reinhardtii}, or artificial nano rods. Suspended in a liquid, the interaction between these particles themselves and between them and the surrounding fluid gives rise to interesting phenomena such as active turbulence. For a high density of microswimmers, such a suspension can be regarded as an example of an active fluid. 

There are several approaches to capture the dynamics of an active fluid starting with adaptions of a model proposed in \textsc{Vicsek} et al.~\cite{vicsek}. Later approaches include the derivation of continuum limit hydrodynamic equations as in \textsc{Toner} \& \textsc{Tu} \cite{tonertu} or adaptions of liquid crystal models as in \textsc{Thampi} \& \textsc{Yeomans} \cite{thampiyeomans}.  Another model we want to mention is the one suggested in \textsc{Wensink} et al.~\cite{wensink}, where a \textsc{Toner--Tu}-like equation is supplemented with a fourth-order \textsc{Swift--Hohenberg} term. 

The model that we want to study here was recently derived in \textsc{Reinken} et al.~\cite{reinken} with the goal of it beeing able to incorporate short-range interactions favoring alignment as well as long-range hydrodynamic interactions. The authors start from overdamped \textsc{Langevin} equations for a generic microscopic model and couple them with a \textsc{Stokes} equation supplemented by an ansatz for the stress tensor.

Our paper is arranged as follows. First, we introduce the equations that are the object of our analysis. Then -- under a simplifying assumption -- we prove the existence of weak solutions via a \textsc{Galerkin} approximation. The main part of the paper is then devoted to prove that such weak solutions obey a relative energy inequality. Such a relative energy inequality can be employed to a variety of uses, such as showing the stability of equilibria (e.g.~in \textsc{Feireisl} \cite{feireisl12}), deriving a posteriori estimates for modelling errors (see \textsc{Fischer} \cite{fischer15}) or as the basis of a generalised concept of solution (e.g.~in \textsc{Lasarzik} \cite{lasarzik19}).

In our case, we will apply the relative energy inequality to prove the weak-strong uniqueness of weak solutions as well as the convergence of those to strong solutions of another problem as one parameter tends to zero.

Throughout this paper, we denote by $c>0$ a generic constant that does not depend on any changing quantity. Furthermore, let $\Omega\subset\mathbb{R}^3$ be a bounded domain of class $\mathscr{C}^2$ and $T>0$ some fixed time. 

Spaces of vector- or matrix-valued functions are denoted by bold letters, e.g.~$\bL^p(\Omega)\coloneqq L^p(\Omega;\mathbb{R}^3)$ for the spaces of integrable functions, and $\bW^{k,p}(\Omega)\coloneqq W^{k,p}(\Omega;\mathbb{R}^3)$ as well as $\bH^{k}(\Omega)\coloneqq W^{k,2}(\Omega;\mathbb{R}^3)$ for \textsc{Sobolev} spaces. The \textsc{Sobolev} space associated with homogeneous \textsc{Dirichlet} boundary conditions is denoted by $\bH^1_0(\Omega):=\operatorname{clos}_{\bH^1}\mathscr{C}_c^\infty(\Omega;\mathbb{R}^3)$, where $\mathscr{C}_c^\infty(\Omega;\mathbb{R}^3)$ denotes the set of infinitely many times differentiable functions that are compactly supported in $\Omega$. We write $\mathscr{C}_{c,\sigma}^\infty(\Omega;\mathbb{R}^3)$ for all such functions that are in addition solenoidal and then denote by $\bL^p_\sigma(\Omega)$ and $\bH^1_{0,\sigma}(\Omega)$ the closure of $\mathscr{C}_{c,\sigma}^\infty(\Omega;\mathbb{R}^3)$ with respect to the standard norms of $\bL^p(\Omega)$ and $\bH^1(\Omega)$, respectively.

The spaces for time-dependent continuous and continuously differentiable functions mapping $[0,T]$ into a \textsc{Banach} space $X$ are denoted by $\mathscr{C}([0,T];X)$ and $\mathscr{C}^1([0,T];X)$, respectively, while the corresponding spaces of \textsc{Bochner}-integrable and weakly differentiable functions are denoted by $L^p(0,T;X)$ and $W^{1,p}(0,T;X)$, respectively. We often omit the time interval $(0,T)$ and the domain $\Omega$ in this notation and write, for example, $L^p(\bW^{1,q})$. Also, for the sake of brevity, we generally do not write out the time dependence of functions under the integral.

For an arbitrary \textsc{Banach} space $X$, we denote its dual by $X^*$ and the dual pairing between $X^*$ and $X$ by $\langle\cdot,\cdot\rangle$. With a slight abuse of notation, however, the dual pairing between the spaces $\bL^p(\Omega)$ and $\bL^q(\Omega)$, where $q$ is the conjugate exponent to $p$, is denoted by $(\cdot,\cdot)$. We use the same notation for the inner product on the \textsc{Hilbert} space $\bL^2(\Omega)\times\bL^2(\Omega)$.

The space $\DAO$ is equipped with the norm $\Vert\Delta\cdot\Vert_{\bL^2}$ (cf. \textsc{Boyer} \& \textsc{Fabrie} \cite[Proposition~IV.5.9]{boyer}).

\section{Model}
We want to study the following equations derived in \textsc{Reinken} et al.~\cite{reinken}. Consider the initial-boundary value problem 
\begin{subequations}\label{eq:model}
\begin{equation}\label{classical}
	\begin{aligned}
	-\Delta\ub + \mu_1\Delta^2\Pb - \gamma_1\Delta\Pb + \lambda_1 (\Pb\cdot\nabla)\Pb+ \nabla \pi_1 &= 0,\\
	\partial_t\Pb + \mu_2\Delta^2\Pb - \gamma_2 \Delta\Pb + \lambda_2 (\Pb\cdot\nabla)\Pb +\alpha \vert \Pb \vert^2\Pb + \beta \Pb\quad& \\+ (\ub\cdot\nabla) \Pb +\kappa \usym\Pb - \uskw\Pb + \nabla\pi_2 &= 0, \\
	\nabla\cdot\ub = \nabla\cdot\Pb &= 0
	\end{aligned}
\end{equation}
on $\Omega\times(0,T)$ and
\begin{equation}\label{eq:initcond}
\begin{array}{rrl}
     \ub= \Pb= \Delta\Pb=0\phantom{\Pb_0}& & \text{ on } \partial\Omega\times(0,T),  \\
     \Pb(\cdot,0)=\Pb_0\phantom{0}& &\text{ in } \Omega,
\end{array}
\end{equation}
\end{subequations}
where $\alpha,\lambda_1,\lambda_2,\mu_1,\mu_2>0$ and $\beta,\gamma_1, \gamma_2,\kappa\in\mathbb{R}$. In contrast to other models, both the influence of the polar ordering of the microswimmers themselves, described by the vector field $\Pb\colon\overline{\Omega}\times[0,T]\rightarrow\mathbb{R}^3$, as well as the one from the velocity field $\ub\colon\overline{\Omega}\times[0,T] \rightarrow\mathbb{R}^3$ of the suspension fluid  is taken into account. The latter is incorporated into the model by coupling the \textsc{Stokes} equation with the equation for the polar ordering parameter $\Pb$, where $\varepsilon$ is a dimensionless parameter that describes the strength of the coupling in the sense that there exist $\varepsilon,\tilde{\lambda_1},\tilde{\mu_1}>0$ and $\tilde{\gamma_1}\in\mathbb{R}$ such that
\begin{equation}\label{eq:coupling}
    \gamma_1=\varepsilon\tilde{\gamma}_1,\;\lambda_1=\varepsilon\tilde{\lambda}_1\;,\mu_1=\varepsilon\tilde{\mu}_1.
\end{equation}
Then for $\varepsilon=0$ both equations decouple (see Section \ref{sec:connect}).

The velocity field of the whole fluid is then given as the sum $\ub+\Pb$ of both vector fields. The third unknown is the pressure $\pi\colon\overline{\Omega}\times[0,T]\rightarrow\mathbb{R}$, where $\pi_1=c_1\pi+c_2 \vert \Pb \vert^2$ and $\pi_2=c_3\pi - c_4\vert \Pb \vert^2$ for certain constants $c_1,c_3,c_4>0$ and $c_2\in\mathbb{R}$. We consider the problem in a variational framework with solenoidal test functions and will not be dealing with the pressure term since it vanishes in the weak formulation.

Unfortunately, we are not able to show the appropriate a apriori estimates to prove the existence of weak solutions for the general case above. Therefore, we restrict ourselves to the simpler case $\kappa=0$.

\section{Existence of weak solutions}\label{sec:existence}

We start by giving the definition of a weak solution to \eqref{eq:model}. Note that the term $\Delta^2\Pb$ occurs in the \textsc{Stokes} equation, meaning that $\Delta\ub$ has at most the regularity of $\Delta^2\Pb$. Hence, the weak formulation of the second equation necessitates a very weak formulation for the \textsc{Stokes} equation. For this, by $A^{-1}\colon\bLs^2(\Omega)\rightarrow\DAO$, we denote the inverse of the \textsc{Stokes} operator, meaning that for $\phib\in \bLs^2(\Omega)$ the vector field $A^{-1}\phib:=\vb$ is the unique solution of
\begin{equation*}
    \begin{array}{rrl}
        -\Delta \bv + \nabla\pi = \phib\phantom{0}& &\text{ in } \Omega,\\
        \nabla\cdot \bv = 0\phantom{\phib}& &\text{ in } \Omega,\\
        \bv=0\phantom{\phib}& &\text{ on } \partial\Omega
    \end{array}
\end{equation*}
in the weak sense. For the existence and properties of this operator, see \textsc{Temam} \cite[Chapter~I, §~2.6]{temam}.
\begin{definition}
[Weak solution]\label{def:weak}
Let $\Pb_0\in\DAO$ be given. A pair $(\ub,\Pb)\in L^2(\bL^2_{\sigma})\times( L^\infty(\bLs^2) \cap L^4(\bLs^4) \cap L^2(\bH^2\cap \bH^1_{0,\sigma}))$ such that $\partial_t\pb\in L^{\frac{4}{3}}((\DA)^*)$ is called weak solution to \eqref{eq:model} if the equations
\begin{equation}\label{eq:weak1}
    \begin{aligned}
	    \intT( \ub , \phib)\dt  - \mu_1 \intT ( \Delta\Pb, \Delta A^{-1}\phib) \dt + \gamma_1 &\intT (\Delta\Pb, A^{-1}\phib ) \dt \\&- \lambda_1 \intT ((\Pb\cdot\nabla)\Pb,A^{-1}\phib) \dt = 0   
\end{aligned}
\end{equation}
and
\begin{equation}\label{eq:weak2}
\begin{aligned}
    \intT\langle \partial_t\Pb,\psib\rangle\dt + \mu_2\intT (&\Delta\Pb,\Delta\psib)\dt + \gamma_2\intT (\nabla\Pb,\nabla\psib) \dt + \lambda_2 \intT ((\Pb\cdot\nabla)\Pb,\psib)\dt \\ &+\alpha  \intT (\vert \Pb \vert^2 \Pb , \psib) \dt+\beta\intT (\Pb,\psib)\dt  +\intT ((\ub\cdot\nabla)\Pb,\psib)\dt\\&+ \frac{1}{2}\intT ((\Pb\cdot\nabla)\psib-(\psib\cdot\nabla)\Pb,\ub)\dt  = 0
\end{aligned}
\end{equation}
hold for all test functions $\phib,\psib\in \mathscr{C}_{c,\sigma}^\infty(\Omega\times(0,T);\mathbb{R}^3)$ and the initial condition $\Pb(0)=\Pb_{0}$ is fulfilled in the weak sense, i.e.,
\begin{equation*}
   \Pb(t)\rightharpoonup\Pb_{0}\text{ in }\bLs^2(\Omega) \text{ as } t\rightarrow 0.
\end{equation*}
\end{definition}

\begin{remark}\label{rem:welldef}
In order to justify that this weak formulation is well-defined, we only consider the nonlinear terms. Since $\Pb$ is divergence-free, the identity 
\begin{equation*}
    (\Pb\cdot\nabla)\Pb=\nabla\cdot(\Pb\otimes\Pb),
\end{equation*}
where $\otimes$ denotes the dyadic product, holds almost everywhere in $\Omega\times(0,T)$. Performing an integration-by-parts and using \textsc{H\"older}'s inequality as well as the continuous embedding $\bH^2(\Omega)\hookrightarrow\bW^{1,6}(\Omega)$ then yields
\begin{equation*}
    \begin{aligned}
        \vert ( (\Pb\cdot\nabla)\Pb,\bw ) \vert\le c \Vert \Pb \Vert_{\bL^{\frac{12}{5}}}^2 \Vert \bw \Vert_{\bH^2}
    \end{aligned}
\end{equation*}
for any $\bw\in\DAO$. Therefore, we can estimate the norm of the functional $\bw\mapsto ( (\Pb\cdot\nabla)\Pb,\bw )$ by
\begin{equation*}
    \begin{aligned}
      \Vert   (\Pb\cdot\nabla)\Pb  \Vert_{(\DA)^*}\le c \Vert \Pb \Vert_{\bL^{\frac{12}{5}}}^2 
    \end{aligned}
\end{equation*}
almost everywhere in $(0,T)$. Hence,
\begin{equation*}
    \begin{aligned}
      \intT\left\vert((\Pb\cdot\nabla)\Pb,A^{-1}\phib)\right\vert\intd t &\le \intT \Vert(\Pb\cdot\nabla)\Pb\Vert_{(\DA)^*} \Vert A^{-1}\phib \Vert_{\DA} \intd t \\
      &\le  c \Vert \Pb \Vert_{L^4(\bL^{\frac{12}{5}})}^2 \Vert \phib \Vert_{L^2(\bL^2)},
    \end{aligned}
\end{equation*}
where we also used the boundedness of $A^{-1}:\bLs^2(\Omega)\rightarrow\DAO$ (see \textsc{Temam} \cite[Section~2.5]{temamnonlin}). For the remaining nonlinear terms, applications of \textsc{H\"older}'s inequality yield
\begin{equation}\label{eq:welldef}
    \begin{aligned}
        \intT  \left\vert((\Pb\cdot\nabla)\Pb,\psib)\right\vert\dt &\le \Vert \Pb \Vert_{L^2(\bL^\infty)}\Vert \nabla\Pb \Vert_{L^2(\bL^6)} \Vert \psib \Vert_{L^\infty(\bL^\frac{6}{5})},\\
        \intT \left\vert(\vert \Pb \vert^2 \Pb , \psib)\right\vert \dt &\le \Vert \Pb \Vert_{L^4(\bL^4)}^3 \Vert \psib \Vert_{L^4(\bL^4)},\\
        \intT \left\vert((\ub\cdot\nabla)\Pb,\psib)\right\vert \dt &\le \Vert \ub \Vert_{L^2(\bL^2)}\Vert \nabla\Pb \Vert_{L^2(\bL^6)} \Vert \psib \Vert_{L^\infty(\bL^3)},\\
        \intT \left\vert((\Pb\cdot\nabla)\psib-(\psib\cdot\nabla)\Pb,\ub)\right\vert\dt &\le \Vert \Pb \Vert_{L^2(\bL^\infty)}\Vert \nabla\psib \Vert_{L^\infty(\bL^2)} \Vert \ub \Vert_{L^2(\bL^2)}\\&\qquad\qquad+\Vert \psib \Vert_{L^\infty(\bL^3)}\Vert \nabla\Pb \Vert_{L^2(\bL^6)} \Vert \ub \Vert_{L^2(\bL^2)}. 
    \end{aligned}
\end{equation}
Together with the continuous embedding $\bH^2(\Omega)\hookrightarrow\bL^\infty(\Omega)$, we see that this notion of a weak solution is well-defined for all test functions $\psib\in L^4(\DA)\cap L^\infty(\bH^1_{0,\sigma})$ and $\phib\in L^2(\bL^2_\sigma)$.
\end{remark}
We can prove the existence of such a solution via a \textsc{Galerkin} approximation.
\begin{theorem}\label{thm:existence}
Let $\Pb_{0}\in \DAO$. Then there exists a weak solution to \eqref{eq:model} in the sense of Definition~\ref{def:weak}.  
\end{theorem}
\begin{proof}
First, we choose a \textsc{Galerkin} basis $\{\bw_m\}_{m\in\mathbb{N}}\subset \bH^2(\Omega)\cap\bH_{0,\sigma}^1(\Omega)$ consisting of eigenfunctions of the \textsc{Stokes} operator $A: \bH^2(\Omega)\cap\bH_{0,\sigma}^1(\Omega)\subset \bL_\sigma^2(\Omega) \rightarrow \bL_\sigma^2(\Omega)$. For any $n\in\mathbb{N}$, let $W_n$ be the finite dimensional subspace of $\bH^2(\Omega)\cap\bH_{0,\sigma}^1(\Omega)$ spanned by $\bw_1,\dots,\bw_n$ and $\Pi_n\colon\bL^2_\sigma(\Omega)\rightarrow W_n$ the $\bL^2(\Omega)$-orthogonal projection onto $W_n$. For the initial value $\Pb_{0,n}\coloneqq\Pi_n \Pb_0$, we are then looking for a solution $(\ubn,\Pbn)\in\mathscr{C}^1([0,T];W_n)\times\mathscr{C}^1([0,T];W_n)$ to the discretised problem
\begin{subequations}\label{eq:disc}
    \begin{equation}\label{eq:disc1}
        \begin{aligned}
             ( \ubn , \bv) -\mu_1(\Delta\Pbn,\Delta A^{-1}  \bv) + \gamma_1 (\Delta\Pbn, A^{-1}\bv ) - \lambda_1 (\Pbn\cdot\nabla)\Pbn,A^{-1}\bv) =0  ,
        \end{aligned}
    \end{equation}
    \begin{equation}\label{eq:disc2}
        \begin{aligned}
              (\partial_t\Pbn,\bw)+\mu_2  (\Delta\Pbn,\Delta\bw) + \gamma_2  (\nabla\Pbn,\nabla\bw)+\lambda_2  ((\Pbn\cdot\nabla)\Pbn,\bw) +\alpha   (\vert \Pbn \vert^2 \Pbn &, \bw)\\+ \beta (\Pbn,\bw) +  ((\ubn\cdot\nabla)\Pbn,\bw)  + \frac{1}{2} ((\Pbn\cdot\nabla)\bw-(\bw\cdot\nabla)\Pbn,\ubn)    &= 0,\\
             \Pbn(0)&=\Pb_{0,n}
          \end{aligned}
    \end{equation}
\end{subequations}
in $(0,T)$ and for all $\bv,\bw\in W_n$. In order to show the existence of a solution to this system, we fix $n\in\mathbb{N}$ and represent each function $\Pbn:[0,T]\rightarrow W_n$ via the corresponding basis of $W_n$, i.e.,
\begin{equation*}
    \Pbn(t)=\sum_{i=1}^n P_{ni}(t)\bw_i.
\end{equation*}
We make use of the inverse $J^{-1}\colon(\bL_\sigma^2(\Omega))^*\rightarrow \bL_\sigma^2(\Omega)$ of the \textsc{Riesz} isomorphism on the $\textsc{Hilbert}$ space $\bL_\sigma^2(\Omega)$ in order to define
\begin{equation}\label{eq:ubyp}
    \ubn = \Pi_n J^{-1} \boldsymbol{g}_{\Pbn},
\end{equation}   
where $\boldsymbol{g}_{\Pbn}\in (\bL_\sigma^2(\Omega))^*$ 
is given as   
\begin{equation*}   
    \langle \boldsymbol{g}_{\Pbn}, \vb \rangle = \mu_1 (\Delta\Pbn,\Delta A^{-1}\bv) - \gamma_1   (\Delta\Pbn,A^{-1}\bv)  +\lambda_1((\Pbn\cdot\nabla)\Pbn, A^{-1}\bv)    
\end{equation*}
for any $\Pb_n\in W_n$ and $\bv\in \bLs^2(\Omega)$. In order to transform the problem to an autonomous system of ordinary differential equations, we let $\by:[0,T]\rightarrow\mathbb{R}^n$ and $\by_0\in\mathbb{R}^n$ be defined as
\begin{equation*}
    (\by(t))_i=P_{ni}(t),\quad (\by_0)_i= (\Pb_{0,n},\bw_i)
\end{equation*}
for $i=1,\dots,n$. Then the discretised problem \eqref{eq:disc} can equivalently be written as 
\begin{align*}
        \by' &= \fb(\by)\quad\text{ on } [0,T],\\
        \by(0)  &= \by_0,
\end{align*}
where the right-hand side $\fb:\mathbb{R}^n\rightarrow\mathbb{R}^n$ is given as
\begin{equation*}
    \begin{aligned}
        \fb(\by)=& -\mu_2  (\Delta\Pbn,\Delta\bw) - \gamma_2  (\nabla\Pbn,\nabla\bw)-\lambda_2  ((\Pbn\cdot\nabla)\Pbn,\bw) -\alpha   (\vert \Pbn \vert^2 \Pbn , \bw)\\&-\beta (\Pbn,\bw) -  ((\ubn\cdot\nabla)\Pbn,\bw)  - \frac{1}{2} ((\Pbn\cdot\nabla)\bw-(\bw\cdot\nabla)\Pbn,\ubn)   
    \end{aligned}
\end{equation*}
for $i=1,\dots,n$, as well as
\begin{equation*}
    \Pbn=\sum_{i=1}^n \by_i\bw_i,
\end{equation*}
and $\ubn$ as in \eqref{eq:ubyp}.
Note that the mass matrix in this case is just the identity since the eigenfunctions $\bw_i$ are orthonormal with respect to the $\bL^2$-inner product. Then $\fb$ does not depend on $t$ and is continuous with respect to $\by$. Therefore, the \textsc{Peano} theorem (see \textsc{Hale} \cite[Chapter~I, Theorem~5.1]{hale}) provides a maximally continued solution $(\ubn,\Pbn)\in\mathscr{C}^1([0,T_n);W_n)\times\mathscr{C}^1([0,T_n);W_n)$ to \eqref{eq:disc}, where $T_n\le T$ may depend on the dimension $n$ of $W_n$. We now derive a priori estimates to show that there is no blow-up of these solution, hence they exist globally in time. 

To this end, we first test equation \eqref{eq:disc2} with $\Pbn$. Since $\ubn$ and $\Pbn$ are both divergence-free, this results in
\begin{equation*}
    \begin{aligned}
             & \frac{1}{2}\frac{\intd}{\intd t} \Vert \Pbn \Vert_{\bL^2}^2+\mu_2  \Vert \Delta\Pbn \Vert_{\bL^2}^2 + \gamma_2  \Vert \nabla\Pbn \Vert_{\bL^2}^2   + \alpha   \Vert \Pbn \Vert_{\bL^4}^4 + \beta \Vert \Pbn \Vert_{\bL^2}^2    = 0 \text{ in } (0,T).
          \end{aligned}
\end{equation*}
Here we also used the identity \[(\partial_t\Pbn,\Pbn)=\frac{1}{2}\frac{\intd}{\intd t} \Vert \Pbn \Vert_{\bL^2}^2 \text{ in } (0,T)\] for continuously differentiable functions. If $\gamma_2,\beta\ge 0$ this already gives the desired a priori estimate for $\pb$. We set
\begin{equation*}
    a^-\coloneqq\begin{cases} 0,& \text{if } a\ge 0\\
    -a,& \text{if } a>0. 
    \end{cases}
\end{equation*}
for any $a\in\mathbb{R}$. Performing an integration-by-parts and using \textsc{Cauchy--Schwarz}'s as well as \textsc{Young}'s inequality then leads to
\begin{equation*}
    \gamma_2^- \Vert \nabla\Pbn \Vert_{\bL^2}^2 \le \frac{( \gamma_2^-)^2}{2\mu_2}\Vert \Pbn \Vert_{\bL^2}^2  +\frac{\mu_2}{2} \Vert \Delta\Pbn \Vert_{\bL^2}^2,
\end{equation*}
and by another application of \textsc{Young}'s inequality as well as the continuous embedding $\bL^4(\Omega)\hookrightarrow\bL^2(\Omega)$, it follows that
\begin{equation*}
    \left(\frac{ (\gamma_2^-)^2}{2\mu_2}+\beta^-\right)\Vert \Pbn \Vert_{\bL^2}^2 \le \frac{\vert \Omega\vert}{2\alpha}\left(\frac{ (\gamma_2^-)^2}{2\mu_2}+\beta^-\right)^2 + \frac{\alpha}{2}\Vert \Pbn \Vert_{\bL^4}^4.
\end{equation*}
By absorbing the norms of $\Pbn$ into the left-hand side, we arrive at
\begin{equation*}
    \begin{aligned}
             & \frac{\intd}{\intd t} \Vert \Pbn \Vert_{\bL^2}^2 +\mu_2  \Vert \Delta\Pbn \Vert_{\bL^2}^2 + \alpha   \Vert \Pbn \Vert_{\bL^4}^4    \le \frac{\vert \Omega\vert}{\alpha}\left(\frac{ (\gamma_2^-)^2}{2\mu_2}+\beta^-\right)^2.
          \end{aligned}
\end{equation*} Integrating the above estimate leads to
\begin{equation}\label{eq:ap1.fin}
    \Vert \Pb_n(t) \Vert_{\bL^2}^2 + \int_0^t\!\left( \mu_2 \Vert \Delta \Pbn \Vert_{\bL^2}^2 + \alpha \Vert \Pbn \Vert^4_{\bL^4}  \right) \intd s \le  \Vert \Pb_{n,0} \Vert_{\bL^2}^2 + \frac{T\vert \Omega\vert}{\alpha}\left(\frac{ (\gamma_2^-)^2}{2\mu_2}+\beta^-\right)^2
\end{equation}
for all $t\in[0,T_n)$. In order to derive an estimate for the velocity field, we now test \eqref{eq:disc1} with $\ubn$, which results in
\begin{equation}\label{eq:ap2.0}
    \begin{aligned}
      \Vert &\ubn \Vert_{\bL^2(\Omega)}^2  = \mu_1  (\Delta\Pbn, \Delta A^{-1}\ubn) -\gamma_1  (\Delta\Pbn, A^{-1}\ubn ) + \lambda_1 ((\Pbn\cdot\nabla)\Pbn,A^{-1}\ubn).
    \end{aligned}
\end{equation}
Just as in Remark~\ref{rem:welldef}, we can derive the estimate
\begin{equation*}
    \begin{aligned}
      \Vert  (\Pb_n\cdot\nabla)\Pb_n  \Vert_{(\DA)^*}\le c \Vert \Pb_n \Vert_{\bL^{\frac{12}{5}}}^2 
    \end{aligned}
\end{equation*} 
and therefore, together with the embedding $\bL^4(\Omega)\hookrightarrow\bL^\frac{12}{5}(\Omega)$,
\begin{equation}\label{eq:ap2.1}
    \begin{aligned}
      ((\Pbn\cdot\nabla)\Pbn,A^{-1}\ub_n)  &\le c \Vert \Pb_n \Vert_{\bL^4}^2 \Vert \ub_n \Vert_{\bL^2}
    \end{aligned}
\end{equation}
almost everywhere in $(0,T_n)$.

By the continuity of $A^{-1}\colon \bLs^2(\Omega)\rightarrow\DAO$, it follows that
\begin{equation}\label{eq:ap2.2}
    \begin{aligned}
      (\Delta\Pb_n,A^{-1}\ub_n)\le c \Vert \Pb_n \Vert_{\bL^2} \Vert \ub_n \Vert_{\bL^2}\le c \Vert \Delta\pbn\Vert_{\bL^2}\Vert \ubn\Vert_{\bL^2}
    \end{aligned}
\end{equation}
and, similarly, we find
\begin{equation}\label{eq:ap2.3}
    \begin{aligned}
        (\Delta\Pbn, \Delta A^{-1}\ubn)\le c \Vert \Delta \Pb_n \Vert_{\bL^2} \Vert \ub_n \Vert_{\bL^2}.
    \end{aligned}
\end{equation}
Applying \textsc{Young}'s inequality in \eqref{eq:ap2.1}, \eqref{eq:ap2.2} and \eqref{eq:ap2.3} and absorbing the terms depending on $\ub_n$ into the left-hand side of \eqref{eq:ap2.0} yields
\begin{equation}\label{eq:extu}
    \Vert \ub_n \Vert_{\bL^2}^2 \le \varepsilon^2c\left( \Vert \Pb_n \Vert_{\bL^4}^4 + \Vert \Delta\Pb_n \Vert_{\bL^2}^2 \right)
\end{equation}
in $(0,T_n)$. Here $\varepsilon>0$ is given by \eqref{eq:coupling}. By integrating this inequality and using estimate \eqref{eq:ap1.fin} as well as the $\bL^2$-stability of the projections $\Pi_n$, we arrive at
\begin{equation}\label{eq:apriori}
    \Vert \Pb_n(t) \Vert_{\bL^2}^2 + \int_0^t\!\left( \Vert \ub_n(s) \Vert_{\bL^2}^2 + \mu_2 \Vert \Delta \Pbn(s) \Vert_{\bL^2}^2 + \alpha \Vert \Pbn(s) \Vert^4_{\bL^4} \right) \intd s \le c (\Vert \Pb_{0} \Vert_{\bL^2}^2+1).
\end{equation}
Since this estimate is uniform in $n$, the discrete solution $(\ubn,\Pbn)$ actually exists on $[0,T]$ (see \textsc{Hale} \cite[Chapter~I, Theorem~5.2]{hale}).

We proceed with deriving an estimate for the time derivative of $\Pbn$. By \eqref{eq:disc2} and \textsc{H\"older}'s inequality, we find
\begin{equation*}
    \begin{aligned}
        \vert \langle \partial_t\Pbn,\bv \rangle \vert =  \vert ( \partial_t\Pbn,\Pi_n\bv ) \vert \le&\mu_2 \Vert \Delta \Pbn \Vert_{\bL^2}\Vert \Delta \Pi_n\bv \Vert_{\bL^2} + \vert \gamma_2 \vert \Vert \nabla\Pbn \Vert_{\bL^\frac{6}{5}}\Vert \nabla\Pi_n\bv \Vert_{\bL^6}\\
        &+ \lambda_2 \Vert \Pbn \Vert_{\bL^2}  \Vert \nabla\Pbn \Vert_{\bL^2} \Vert \Pi_n\bv \Vert_{\bL^\infty} + \alpha \Vert \Pbn \Vert_{\bL^3}^3\Vert \Pi_n\bv \Vert_{\bL^\infty}\\ 
        &+ \vert \beta \vert \Vert \Pbn \Vert_{\bL^1} \Vert \Pi_n\bv \Vert_{\bL^\infty} +\frac{3}{2}\Vert \ubn \Vert_{\bL^2}  \Vert \Pbn \Vert_{\bL^3} \Vert \nabla\Pi_n\bv \Vert_{\bL^6}\\ &+ \frac{1}{2}\Vert \Pi_n\bv \Vert_{\bL^\infty}  \Vert \nabla \Pbn \Vert_{\bL^2} \Vert \ubn \Vert_{\bL^2} 
    \end{aligned}
\end{equation*}
for all $\bv\in\DAO$. From the contiuous embeddings $\bH^2(\Omega)\hookrightarrow\bW^{1,6}(\Omega)\hookrightarrow\bL^\infty(\Omega)$ and the $\bH^2$-stability of the orthogonal projections $\Pi_n$ (cf. \textsc{M\'alek}, \textsc{Ne\v cas}, \textsc{Rokyta} \& \textsc{R\r u\v zi\v cka} \cite[Appendix, Theorem~4.11 and Lemma~4.26]{malek} together with \textsc{Boyer} \& \textsc{Fabrie} \cite[Proposition~III.3.17]{boyer}), it follows that
\begin{equation*}
    \begin{aligned}
    \Vert \partial_t\Pbn \Vert_{(\DA)^*} \le c\big(\Vert \Delta \Pbn \Vert_{\bL^2}+ \Vert \Pbn \Vert_{\bL^2}  \Vert \nabla\Pbn \Vert_{\bL^2} + \Vert \Pbn \Vert_{\bL^4}^3 + \Vert &\ubn \Vert_{\bL^2}  \Vert \Pbn \Vert_{\bL^3} \\&+  \Vert \nabla \Pbn \Vert_{\bL^2} \Vert \ubn \Vert_{\bL^2}   \big).
    \end{aligned}
\end{equation*}
Now, using the embedding $\bH^1_0(\Omega)\hookrightarrow\bL^6(\Omega)$, we find
\begin{equation*}
    \Vert \Pbn \Vert_{\bL^3} \le \Vert \Pbn \Vert_{\bL^2}^\frac{1}{2} \Vert \Pbn \Vert_{\bL^6}^\frac{1}{2} \le c \Vert \Pbn \Vert_{\bL^2}^\frac{1}{2} \Vert \nabla\Pbn \Vert_{\bL^2}^\frac{1}{2}. 
\end{equation*}
Together with
\begin{equation*}
    \Vert \nabla\Pbn \Vert_{\bL^2}^2 = (\nabla\pbn, \nabla\pbn) = -( \pbn,\Delta\pbn) \le \Vert \Pbn \Vert_{\bL^2} \Vert \Delta \Pbn \Vert_{\bL^2}
\end{equation*}
as well as \textsc{Young}'s inequality, we arrive at
\begin{equation*}
    \begin{aligned}
    \Vert \partial_t\Pbn \Vert_{(\DA)^*} \le c\big( \Vert \Delta \Pbn& \Vert_{\bL^2} +  \Vert \Pbn \Vert_{\bL^2}  \Vert \nabla\Pbn \Vert_{\bL^2} +  \Vert \Pbn \Vert_{\bL^4}^3 \\ &\Vert \Pbn \Vert_{\bL^2}^\frac{3}{4}(\Vert \ubn \Vert_{\bL^2}^\frac{5}{4} +   \Vert \Delta \Pbn \Vert_{\bL^2}^\frac{5}{4}) +  \Vert \Pbn \Vert_{\bL^2}^\frac{1}{2}(\Vert \ubn \Vert_{\bL^2}^\frac{3}{2} +   \Vert \Delta \Pbn \Vert_{\bL^2}^\frac{3}{2})  \big).
    \end{aligned}
\end{equation*}
Taking the estimate above to the power $\frac{4}{3}$, integrating and applying the a priori estimate \eqref{eq:apriori} then yields
\begin{equation*}
    \int_0^t\! \Vert \partial_t\Pbn \Vert_{(\DA)^*}^\frac{4}{3} \intd s \le c (\Vert \Pb_0 \Vert_{\bL^2}^2+1).
\end{equation*}
Together with \eqref{eq:apriori}, we can extract (not relabled) subsequences $\{\Pbn\}$ and $\{\ubn\}$ such that
\begin{subequations}
\begin{align}
    \phantom{\partial_t}\ub_n &\rightharpoonup \ub\phantom{\partial_t}\quad \text{ in } L^2(0,T;\bL^2_{\sigma})\\
	\phantom{\partial_t}\Pb_n &\overset{*}{\rightharpoonup} \Pb\phantom{\partial_t} \quad\text{ in } L^\infty(0,T;\bLs^2)\label{eq:convP1}\\
	\phantom{\partial_t}\Pb_n &\rightharpoonup \Pb\phantom{\partial_t} \quad\text{ in } L^4(0,T;\bLs^4)\\
	\phantom{\partial_t}\Pb_n &\rightharpoonup \Pb\phantom{\partial_t} \quad\text{ in } L^2(0,T;\bH^2\cap \bH^1_{0,\sigma})\\
	\partial_t\Pb_n &\rightharpoonup \partial_t\Pb \quad\text{ in } L^{\frac{4}{3}}(0,T;(\bH^2\cap\bH^1_{0,\sigma})^*)\label{eq:convP4}
\end{align}
\end{subequations}
as $n\rightarrow\infty$. Let us show that $(\ub,\pb)$ is a weak solution to \eqref{eq:model} in the sense of Definition \eqref{def:weak}. Since $W_m\subset W_n$ for $m\le n$, \eqref{eq:disc} also holds true for all $\phi(t)\bv,\psi(t)\bw$, where $t\in(0,T)$, $\phi,\psi\in \mathscr{C}_c^\infty(0,T)$ and $\bv,\bw\in W_m$. We fix $m$ and let $n$ go to $\infty$. In view of the compact embedding $\DAO\overset{c}{\hookrightarrow}\bH^1_{0,\sigma}(\Omega)$, the \textsc{Lions--Aubin} lemma (see \textsc{Lions} \cite[Th\'{e}or\`{e}me~1.5.2]{lionsquel}) provides the strong convergence 
\begin{align}\label{eq:convP5}
	\Pb_n &\rightarrow \Pb \quad\text{ in } L^2(0,T;\bH^1_{0,\sigma}),
\end{align}
passing to a subsequence if necessary. This in turn allows us, after integrating, to pass to the limit also in the nonlinear terms of \eqref{eq:disc}. Hence, we come up with
\begin{equation*}
    \begin{aligned}
	\intT( \ub , \phi\bv)\dt  - \mu_1 \intT ( \Delta\Pb, \Delta A^{-1}(\phi\bv)) \dt + \gamma_1 &\intT (\Delta\Pb, A^{-1}(\phi\bv) ) \dt \\&- \lambda_1 \intT ((\Pb\cdot\nabla)\Pb,A^{-1}(\phi\bv)) \dt = 0   
\end{aligned}
\end{equation*}
and
\begin{equation*}
\begin{aligned}
    \intT\langle \partial_t\Pb,\psi\bw\rangle\dt + \mu_2\intT (&\Delta\Pb,\Delta(\psi\bw))\dt + \gamma_2\intT (\nabla\Pb,\nabla(\psi\bw)) \dt + \lambda_2 \intT ((\Pb\cdot\nabla)\Pb,\psi\bw)\dt \\ &+\alpha  \intT (\vert \Pb \vert^2 \Pb , \psi\bw) \dt+\beta\intT (\Pb,\psi\bw)\dt  +\intT ((\ub\cdot\nabla)\Pb,\psi\bw)\dt\\&+ \frac{1}{2}\intT ((\Pb\cdot\nabla)(\psi\bw)-((\psi\bw)\cdot\nabla)\Pb,\ub)\dt  = 0
\end{aligned}
\end{equation*}
for all $\phi,\psi\in \mathscr{C}_c^\infty(0,T)$ and $\bv,\bw\in W_m$ with arbitrary $m$ and thus for all $\bv,\bw\in\DAO$ since $\cup_{m\in\mathbb{N}}W_m$ is dense in $\DAO$. It remains to show that $\pb(0)=\pb_0$. In order to show this, we observe that
\begin{equation*}
    \Pb\in W^{1,\frac{4}{3}}(0,T;(\DA)^*)\hookrightarrow \mathscr{C}([0,T];(\DA)^*)\hookrightarrow \mathscr{C}_w([0,T];(\DA)^*),
\end{equation*}
making point evaluations of $\Pb$ in the dual pairing well-defined. Since also $\pb\in L^\infty(0,T;\bLs^2)$, it follows from a well-known lemma (see \textsc{Lions \& Magenes} \cite[Chapter 3, Lemma 8.1]{lionsmagenes}) that
\begin{equation*}
    \pb\in \mathscr{C}_w([0,T];\bLs^2)
\end{equation*}
and in particular
\begin{equation}\label{eq:initialvalue1}
    \Pb(t)\rightharpoonup\Pb(0) \text{ in } \bLs^2(\Omega)
\end{equation}
as $t\rightarrow 0$. Furthermore, for all $\bv\in\DAO$ and $\varphi(t)=(T-t)/T$, there holds
\begin{equation*}
    \langle \Pbn(0)-\Pb(0),\bv\rangle = - \intT\left( \langle \Pbn(t)-\Pb(t),\vb \varphi'(t)\rangle + \langle \partial_t\Pbn(t)-\partial_t\Pb(t),\vb \varphi(t)\rangle \right)\intd t\rightarrow 0 
\end{equation*}
as $n\rightarrow\infty$, by the convergence in \eqref{eq:convP1} and \eqref{eq:convP4}. This implies
\begin{equation}\label{eq:initialvalue}
    \Pbn(0)\rightarrow\Pb(0) \text{ in }(\DAO)^*.
\end{equation}
Since also 
\begin{equation*}
    \Pbn(0)=\Pi_n\Pb_0\rightarrow \Pb_0\in \bL^2_\sigma(\Omega),
\end{equation*}
it follows that
\begin{equation}\label{eq:intialvalue2}
    \Pb(0)=\Pb_0 \text{ in } (\DAO)^*.
\end{equation}
By \eqref{eq:initialvalue1} and \eqref{eq:intialvalue2}, it finally follows that
\begin{equation*}
    \Pb(t)\rightharpoonup\Pb_0 \text{ in } \bLs^2(\Omega)
\end{equation*}
as $t\rightarrow 0$.
\end{proof}
From the above proof of existence, we can easily see that the following energy inequality as well as an estimate for $\ub$ hold true.
\begin{lemma}[Energy inequality]\label{lem:energyineq}
Let $(\ub,\Pb)$ be the weak solution to \eqref{eq:model} in the sense of Definition~\ref{def:weak} that was constructed in the proof of Theorem \eqref{thm:existence}. Then the estimates
\begin{equation}\label{eq:energy}
    \begin{aligned}
         \frac{1}{2}\Vert \Pb(t) \Vert_{\bL^2}^2 + \int_0^t\! \left(  \mu_2  \Vert \Delta\Pb \Vert_{\bL^2}^2+ \gamma_2  \Vert \nabla\Pb \Vert_{\bL^2}^2  + \alpha   \Vert \Pb \Vert_{\bL^4}^4 +\beta \Vert \Pb \Vert_{\bL^2}^2    \right)&  \intd s\le \frac{1}{2}\Vert \Pb_0 \Vert_{\bL^2}^2,
    \end{aligned}
\end{equation}
\begin{equation}\label{eq:energy1}
    \Vert \Pb(t) \Vert_{\bL^2}^2 + \int_0^t\!\left( \mu_2 \Vert \Delta \Pb \Vert_{\bL^2}^2 + \alpha \Vert \Pb \Vert^4_{\bL^4}  \right) \intd s \le  \Vert \Pb_{0} \Vert_{\bL^2}^2 + \frac{T\vert \Omega\vert}{\alpha}\left(\frac{ (\gamma_2^-)^2}{2\mu_2}+\beta^-\right)^2
\end{equation}
and
\begin{equation}\label{eq:energy2}
\begin{aligned}
\int_0^t\Vert  \ub \Vert_{\bL^2}^2 \intd s \le \varepsilon^2c (\Vert \Pb_0 \Vert_{\bL^2}^2 +1)
\end{aligned}
\end{equation}
hold for almost all $t\in(0,T)$.
\end{lemma}
\begin{proof}
  As seen in the proof of Theorem~\ref{thm:existence}, we obtain the relative energy inequality for the discrete solution $\Pb_n$ by testing the discretised equation \eqref{eq:disc2} with it. The estimate then also holds in the limit, as $n\rightarrow\infty$, by the sequential lower semicontinuity of the norm and employing the strong convergence \eqref{eq:convP5} for those terms with possibly negative coefficients. In the same fashion as we derived the a priori estimate \eqref{eq:ap1.fin}, we then deduce \eqref{eq:energy1} from \eqref{eq:energy}. Finally, the third estimate can be obtained by testing the very weak formulation of the \textsc{Stokes} equation \eqref{eq:weak1} with the solution $\ub$, applying the same steps as to reach \eqref{eq:extu} and using \eqref{eq:energy1}.
\end{proof}

\section{A relative energy inequality}\label{sec:relenergy}

In the previous section, we have constructed a weak solution $(\ub,\pb)$ to \eqref{eq:model} for which the time derivative $\partial_t\pb$ of the polar ordering field is only in $L^\frac{4}{3}(0,T;(\DA)^*)$. Similarly as in the case of the incompressible \textsc{Navier--Stokes} equations, we are therefore not allowed to test the weak formulation \eqref{eq:weak2} with the solution $\pb\in L^2(0,T;\DA)\cap L^\infty(0,T;\bLs^2) \cap L^4(0,T;\bLs^4)$ itself, meaning that the obvious path for proving uniqueness of a solution is blocked. A way to at least partially compensate for the missing uniqueness is the concept of weak-strong uniqueness. A weak solution possesses the weak-strong uniqueness property if it coincides with the unique strong solution as long as it exists. The fact that \textsc{Leray--Hopf} solutions to the \textsc{Navier--Stokes} equations have this property, is well-known (see, e.g., \textsc{Leray} \cite{leray34}, \textsc{Prodi} \cite{prodi59} and \textsc{Serrin} \cite{serrin63}). The goal of this section is now to prove the same for weak solutions to \eqref{eq:model} that fulfill the energy inequality \eqref{eq:energy}. The main argument here will be a suitable relative energy inequality. 

We first define two functionals: the relative energy $\mathscr{E}\colon\bLs^2(\Omega)\times\bLs^2(\Omega)\rightarrow\mathbb{R}_{\ge0}$ by
\begin{align*}
    	\mathscr{E}[\Pb|\Pbt]=& \frac{1}{2}\Vert \Pb -\Pbt\Vert_{\bL^2}^2
\end{align*}
and the relative dissipation $\mathscr{W}: \DAO \times \DAO\rightarrow\mathbb{R}_{\ge0} $ by
\begin{align*}
    	\mathscr{W}[\Pb|\Pbt]= \mu_2\Vert \Delta\Pb -\Delta\Pbt\Vert_{\bL^2}^2 + \gamma_2^+ \Vert \nabla \Pb - \nabla\Pbt\Vert_{\bL^2}^2 +\alpha\Vert \Pb &-\Pbt\Vert_{\bL^4}^4 +\beta^+\Vert \Pb -\Pbt\Vert_{\bL^2}^2,
\end{align*}
where for any real number $a$ we set $a^+\coloneqq a+a^-$. Both the relative energy and the relative dissipation are supposed to provide a way to measure the ``distance'' between two functions. Hence, we have to ensure their nonnegativity. Since we do not know the sign of $\beta$ and $\gamma_2$, we therefore have to correct those possibly negative terms in the relative dissipation $\mathscr{W}$ by taking only their positive parts $\gamma_2^+$ and $\beta^+$, respectively. For $\pb,\pbt\in L^\infty(\bLs^2)\cap L^4(\bLs^4)\cap \bL^2(\DA)$, setting
\begin{equation*}
    \mathscr{E}[\pb,\pbt](t)\coloneqq\mathscr{E}[\pb(t),\pbt(t)] \quad\text{and}\quad \mathscr{W}[\pb,\pbt](t)\coloneqq\mathscr{W}[\pb(t),\pbt(t)],\qquad t\in[0,T],
\end{equation*}
we conclude that $\mathscr{E}[\Pb|\Pbt]\in L^\infty(0,T)$ and $\mathscr{W}[\Pb|\Pbt]\in L^1(0,T)$.
Furthermore, we introduce the mapping \[\mathscr{K}\colon L^2(0,T;\bL^\infty)\cap L^\infty(0,T;\bLs^\frac{6}{5})\cap L^2(0,T;\bW^{1,\infty})\cap L^4(0,T;\bLs^4)\rightarrow L^1(0,T)\] with
	\begin{align}\label{eq:defk}
    	\mathscr{K}[\Pbt](t)\coloneqq \varepsilon c\large(1+ \Vert\nabla\Pbt(t)\Vert_{\bL^\infty}^2+ \Vert \Pbt \Vert_{L^\infty(\bL^\frac{6}{5})}^2\Vert\nabla\Pbt(t)\Vert_{\bL^\infty}^2 +\Vert\Pbt(t)\Vert_{\bL^4}^4 \large),
	\end{align}
where $\varepsilon>0$ is explained by \eqref{eq:coupling} and $c>0$ is a constant only depending on the coefficients $\alpha,\beta,\tilde{\gamma}_1,\gamma_2,\tilde{\lambda}_1,\lambda_2,\tilde{\mu}_1,\mu_2$ and the domain $\Omega$. 

Finally, we introduce an operator encoding the strong formulation of the equations. Let
\begin{equation*}
\mathscr{R}\colon  L^2(\DA) \times (W^{1,2}(\bLs^2) \cap L^6(\bLs^6)\cap L^2(\bH^4\cap \bH^1_{0,\sigma}))\rightarrow L^2(\bL^2)\times L^2(\bL^2)
\end{equation*}
be defined as
\begin{equation*}
\mathscr{R}[\ubt,\Pbt]=\left(\begin{array}{c} \mathscr{R}_1[\ubt,\Pbt]\\
\mathscr{R}_2[\ubt,\Pbt] \end{array}\right), 
\end{equation*}
where
\begin{align*}
\mathscr{R}_1[\ubt,\Pbt]=&-\Delta\ubt+\mu_1\Delta^2\Pbt -\gamma_1\Delta\Pbt + \lambda_1 (\Pbt\cdot\nabla)\Pbt
\end{align*}
and
\begin{align*}
	\mathscr{R}_2[\ubt,\Pbt]=&\partial_t\Pbt + \mu_2\Delta^2\Pbt + (\ubt\cdot\nabla) \Pbt - \gamma_2 \Delta\Pbt  +\lambda_2 (\Pbt\cdot\nabla)\Pbt  +\alpha \vert \Pbt \vert^2\Pbt +\beta \Pbt - \utskw\Pbt.
\end{align*}
Multiplying $\mathscr{R}_2[\ubt,\Pbt]$ by $\Pbt$ and integrating in space and time yields
\begin{equation}\label{eq:shiftenerg}
    \begin{aligned}
         \frac{1}{2}\Vert \Pbt(t) \Vert_{\bL^2}^2 + \int_0^t\! \Big(\mu_2  \Vert \Delta\Pbt \Vert_{\bL^2}^2+ \gamma_2  &\Vert \nabla\Pbt \Vert_{\bL^2}^2 + \alpha   \Vert \Pbt \Vert_{\bL^4}^4 + \beta \Vert \Pbt \Vert_{\bL^2}^2 \Big)  \intd s\\ &= \frac{1}{2}\Vert \Pbt(0) \Vert_{\bL^2}^2 + \int_0^t\!(\mathscr{R}_2[\ubt,\Pbt],\Pbt)\intd s.
    \end{aligned}
\end{equation}
Note that if we choose $(\ubt,\Pbt)$ as a strong solution to \eqref{eq:model} then $\mathscr{R}[\ubt,\Pbt]$ vanishes and the identity \eqref{eq:shiftenerg} becomes the energy equality in \eqref{eq:energy}. In this sense, we can think about \eqref{eq:shiftenerg} as a shifted energy equality.

We now can state the desired relative energy inequality in a way that allows us to compare a weak solution that fulfills the energy inequality \eqref{eq:energy} with an arbitrary pair of functions with additional regularity, which cannot be expected for weak solutions.
\begin{theorem}\label{thm:releng}
Let $(\ub,\Pb)$ be a weak solution to \eqref{eq:model} in the sense of Definition~\ref{def:weak} that additionally fulfills the energy inequality \eqref{eq:energy}. Then for any pair $(\ubt,\Pbt)$ of test functions such that
\begin{equation*}
    \begin{aligned}
        \ubt\in L^2(\DA),\; \Pbt\in W^{1,2}(\bLs^2) \cap L^6(\bLs^6)\cap L^2(\bH^4\cap \bH^1_{0,\sigma}),
    \end{aligned}
\end{equation*}
the relative energy inequality
\begin{equation}\label{eq:relengery}
    \begin{aligned}
        \mathscr{E}[\Pb|\Pbt](t)+\frac{1}{2}\int_{0}^{t}\!\mathscr{W}[\Pb|\Pbt]\intd s \le \mathscr{E}&[\Pb_0|\Pbt(0)] +  \int_{0}^{t}\!\mathscr{K}[\Pbt]\mathscr{E}[\Pb|\Pbt]\intd s \\&+ \int_{0}^{t}\!\left(\mathscr{R}[\ubt,\Pbt],\left(\begin{array}{c} A^{-1}(\ubt-\ub)\\ \Pbt-\Pb \end{array}\right)\right)\intd s
    \end{aligned}
\end{equation}
holds for almost all $t\in(0,T)$.
\end{theorem}
In order to prove the relative energy inequality above, we first need two lemmata: an integration-by-parts formula and a lower bound for the relative dissipation $\mathscr{W}$ in terms of the distance of $\ub$ and $\ubt$.
\begin{lemma}[Integration-by-parts formula]\label{lem:intby}
For all \[\Pb\in L^\infty(\bL^2) \cap L^4(\bL^4) \cap L^2(\bH^2\cap \bH^1_{0,\sigma})\] such that $\partial_t\pb\in L^{\frac{4}{3}}((\DA)^*)$ and all \[\Pbt\in L^4(\DA)\cap W^{1,2}((\DA)^*)\] there holds
\begin{equation}\label{eq:intbyp}
    \int_{s}^{t}\!\left( \langle \partial_t\Pb,\Pbt \rangle + \langle \Pb,\partial_t\Pbt\rangle\right)\intd \tau=(\Pb(t),\Pbt(t))-(\Pb(s),\Pbt(s))
\end{equation}
for all $s,t\in[0,T]$.
\end{lemma}
\begin{proof}
 First, we choose sequences $\{\Pb_n\},\{\Pbt_n\}\subset \mathscr{C}^1([0,T];\DA)$ such that
\begin{equation*}
    \begin{aligned}
        \Pb_n\rightarrow\Pb\text{ in } L^2(\DA)\cap W^{1,\frac{4}{3}}((\DA)^*),\\
        \Pbt_n\rightarrow\Pbt\text{ in } L^4(\DA)\cap W^{1,2}((\DA)^*).
    \end{aligned}
\end{equation*}
Such sequences exist since $\mathscr{C}^1([0,T];\DA)$ is dense in both $L^2(\DA)\cap W^{1,\frac{4}{3}}((\DA)^*)$ and $L^4(\DA)\cap W^{1,2}((\DA)^*)$. By the higher time-regularity of the approximating sequence, we find
\begin{equation}\label{eq:intbyp1}
    -\int_0^T\!\varphi'(\Pb_n,\Pbt_n) \intd t= \int_{0}^{T}\! \varphi\left(\langle\partial_t\Pb_n,\Pbt_n\rangle + (\Pb_n,\partial_t\Pbt_n))\right)\intd t
\end{equation}
for all $\varphi\in\mathscr{C}_c^\infty(0,T)$. 
From the estimates
\begin{equation*}
    \begin{aligned}
         \int_{0}^{T}\! \vert \langle\partial_t\Pb,\Pbt\rangle -& \langle\partial_t\Pb_n,\Pbt_n\rangle  \vert\intd t  \\ \le& \Vert \partial_t\Pb - \partial_t\Pb_n\Vert_{L^\frac{4}{3}((\DA)^*)} \Vert \Pbt \Vert_{L^4(\DA)} \\&+ \Vert \partial_t\Pb_n\Vert_{L^\frac{4}{3}((\DA)^*)} \Vert \Pbt - \Pbt_n \Vert_{L^4(\DA)}
    \end{aligned}    
\end{equation*}
and
\begin{equation*}
    \begin{aligned}
        \int_{0}^{T}\! \vert(\Pb,\partial_t\Pbt) -& (\Pb_n,\partial_t\Pbt_n) \vert  \intd t  \\ \le& \Vert \Pb\Vert_{L^2(\DA)} \Vert \partial_t\Pbt -\partial_t\Pbt_n \Vert_{L^2((\DA)^*)} \\&+ \Vert \Pb -\Pb_n\Vert_{L^2(\DA)} \Vert \partial_t\Pbt_n \Vert_{L^2((\DA)^*)}
    \end{aligned}    
\end{equation*}
as well as 
\begin{align*}
    \int_0^T\!\vert (\Pb,\Pbt) &- (\Pb_n,\Pbt_n)\vert \intd t \\ \le& \Vert \Pb\Vert_{L^2(\bL^2)} \Vert \Pbt -\Pbt_n \Vert_{L^2(\bL^2)} + \Vert \Pb -\Pb_n\Vert_{L^2(\bL^2)} \Vert \Pbt_n \Vert_{L^2(\bL^2)},
\end{align*}
together with the embedding $L^4(\DA)\hookrightarrow L^2(\DA) \hookrightarrow L^2(\bL^2)$, we conclude that the equality \eqref{eq:intbyp1} holds also in the limit as $n\rightarrow\infty$, i.e.,
\begin{equation}\label{eq:intbypartsform}
    -\int_0^T\!\varphi'(\Pb,\Pbt) \intd t= \int_{0}^{T}\! \varphi\left(\langle\partial_t\Pb,\Pbt\rangle + (\Pb,\partial_t\Pbt))\right)\intd t
\end{equation}
for all $\phi\in \mathscr{C}_c^\infty(0,T)$. Note that $\Vert \partial_t\Pb_n\Vert_{L^\frac{4}{3}((\DA)^*)}$, $\Vert \partial_t\Pbt_n \Vert_{L^2((\DA)^*)}$ and $\Vert \Pbt_n \Vert_{L^2(\bL^2)}$ are bounded uniformly in $n$. Using the same embedding as above, we find 
\begin{equation*}
    \intT \vert (\pb,\pbt) \vert \intd t \le c \Vert \pb \Vert_{L^2(\DA)} \Vert \pbt \Vert_{L^2(\DA)}. 
\end{equation*}
Since also 
\begin{align*}
    \intT\vert \langle\partial_t&\Pb,\Pbt\rangle + (\Pb,\partial_t\Pbt) \vert \intd t\\ &\le \Vert \partial_t\pb \Vert_{L^\frac{4}{3}((\DA)*)} \Vert \pbt \Vert_{L^4(\DA)} + \Vert \pb \Vert_{L^2(\DA)}\Vert \partial_t \pbt \Vert_{L^2((\DA)^*)}, 
\end{align*}
it follows from \eqref{eq:intbypartsform} that the mapping
\begin{equation}\label{eq:mapping}
    t\mapsto (\Pb(t),\Pbt(t))
\end{equation}
is in $W^{1,1}(0,T)$ since the weak derivative
\begin{equation*}
    t\mapsto \langle\partial_t\Pb(t),\Pbt(t)\rangle + (\Pb(t),\partial_t\Pbt(t))
\end{equation*}
is in $L^1(0,T)$. In particular, the function \eqref{eq:mapping} is absolutely continuous on $[0,T]$ and we can apply the fundamental theorem of calculus to finally prove the statement.
\end{proof}
\begin{lemma}\label{lem:coerc}
Let $(\ub,\Pb)$ be a weak solution to \eqref{eq:model} in the sense of Definition~\ref{def:weak} and $\varepsilon$ be chosen as in \eqref{eq:coupling}. Then there exists a constant $C>0$ such that for all 
\begin{equation*}
    \begin{aligned}
        \ubt&\in L^2(\DA),\; \Pbt\in W^{1,2}(\bLs^2) \cap L^6(\bLs^6)\cap L^2(\bH^4\cap \bH^1_{0,\sigma})
    \end{aligned}
\end{equation*}
there holds
\begin{equation}\label{eq:coerc1}
    \begin{aligned}
       \frac{1}{2}\int_0^t \Vert \ub - \ubt \Vert^2_{\bL^2} \le 
     \varepsilon C (1+\Vert \Pbt \Vert_{L^\infty(\bL^{\frac{6}{5}})}^2)\int_0^t\!\mathscr{W}[\Pb,\Pbt]\intd s +\int_0^t\!(\mathscr{R}_1[\ubt,\Pbt],A^{-1}(\ubt-\ub))\intd s
    \end{aligned}
\end{equation}
for all $t\in[0,T]$.
\end{lemma}
Note that we could allow for a larger class of test functions $(\ubt,\Pbt)$ in this lemma by shifting some of the derivatives contained in $\mathcal{R}_1[\ubt,\Pbt]$ onto $A^-1(\ubt-\ub)$ using an integration-by-parts formula. Since we only need to be able to apply estimate \eqref{eq:coerc1} to strong solutions, we restrict ourselves here to the regularity assumptions above.
\begin{proof}
Using equation \eqref{eq:weak1} as well as adding and subtracting $(\mathscr{R}_1[\ubt,\Pbt],A^{-1}(\ubt-\ub))$ yields
  \begin{equation*}
      \begin{aligned}
         \int_0^t\!\Vert \ub - \ubt &\Vert_{\bL^2}^2 \intd s 
        \\=& \mu_1\int_0^t\!  (\Delta\Pbt - \Delta\Pb, \Delta A^{-1}(\ub - \ubt ))\intd s  -\gamma_1 \int_0^t\!  (\Delta\Pbt - \Delta\Pb, A^{-1}(\ub - \ubt ) ) \intd s\\
        &+\lambda_1\int_0^t\!\left( ((\Pbt\cdot\nabla)\Pbt,A^{-1}(\ub - \ubt )  - ((\Pb\cdot\nabla)\Pb,A^{-1}(\ub - \ubt ) \right)\intd s \\
        &+ \int_0^t\! (\mathscr{R}_1[\ubt,\Pbt],A^{-1}(\ubt-\ub)) \intd s\\
        \eqqcolon& \mu_1 J_1 - \gamma_1 J_2 + \lambda_1 J_3 + \int_0^t\! (\mathscr{R}_1[\ubt,\Pbt],A^{-1}(\ubt-\ub)) \intd s.
      \end{aligned}
  \end{equation*}
Since $\Pb$ and $\Pbt$ are divergence-free, we can rewrite $J_1$, perform an integration-by-parts, apply \textsc{H\"older}'s inequality and use the continuous embeddings $\bH^2(\Omega)\hookrightarrow\bW^{1,6}(\Omega)\hookrightarrow\bL^\infty(\Omega)$ and $\bL^4(\Omega)\hookrightarrow\bL^{\frac{12}{5}}(\Omega)$ as well as the continuity of $A^{-1}:\bLs^2(\Omega)\rightarrow\DAO$ to obtain
   \begin{equation*}
      \begin{aligned}
        \vert J_3 \vert &=  \left\vert \int_0^t\! (\Pbt\otimes (\Pbt - \Pb) - (\Pbt-\Pb)\otimes (\Pbt - \Pb) + (\Pbt-\Pb)\otimes \Pbt, \nabla A^{-1}(\ubt-\ub))  \intd s \right\vert \\
        &\le c\int_0^t\! \Vert \Pbt\otimes (\Pbt - \Pb) - (\Pbt-\Pb)\otimes (\Pbt - \Pb) + (\Pbt-\Pb)\otimes \Pbt \Vert_{\bL^{\frac{6}{5}}} \Vert \ubt - \ub \Vert_{\bL^2} \intd s\\
        &\le c\int_0^t\! \left( 2 \Vert \Pbt \Vert_{\bL^{\frac{6}{5}}} \Vert \Pbt - \Pb\Vert_{\bL^\infty} +  \Vert \Pbt-\Pb \Vert_{\bL^{\frac{12}{5}}}^2  \right)\Vert \ubt - \ub \Vert_{\bL^2} \intd s \\
        &\le c \int_0^t\! ( 1+ \Vert \Pbt \Vert_{\bL^{\frac{6}{5}}})\left(  \Vert \Delta\Pbt - \Delta\Pb\Vert_{\bL^2}+ \Vert \Pbt-\Pb \Vert_{\bL^4}^2 \right) \Vert \ubt - \ub \Vert_{\bL^2} \intd s.
      \end{aligned}
  \end{equation*}
  In a similar way, we can estimate the term $J_1$ by
  \begin{equation*}
     \vert J_1 \vert \le c\int_0^t\!  \Vert \Delta\Pbt - \Delta\Pb \Vert_{\bL^2}   \Vert  \ub - \ubt  \Vert_{\bL^2} \intd s. 
 \end{equation*}
 and the term $J_2$ by
  \begin{equation*}
      \begin{aligned}
         \vert J_2 \vert &\le \int_0^t\!   \Vert \Delta\Pbt - \Delta\Pb \Vert_{\bL^2} \Vert  A^{-1}(\ub - \ubt ) \Vert_{\bL^2} 
         \intd s\\
        & \le c\int_0^t\!  \Vert \Delta\Pbt - \Delta\Pb \Vert_{\bL^2}   \Vert  \ub - \ubt  \Vert_{\bL^2} \intd s.
      \end{aligned}
  \end{equation*}
 Using the above estimates in \eqref{eq:coerc1}, applying \textsc{Young}'s inequality and absorbing the terms depending on $\ub$ and $\ubt$ into the left-hand side then proves the assertion. 
\end{proof}
We are now able to prove the relative energy inequality.
\begin{proof}[Proof of Theorem~\ref{thm:releng}]
We start by using the integration-by-parts formula from Lemma~\ref{lem:intby} to obtain
\begin{align}\label{eq:releng0}
	\mathscr{E}[\Pb|\Pbt](t) =& \frac{1}{2} \Vert \Pb(t)\Vert_{\bL^2}^2  -(\Pb_0,\Pbt(0))- \int_{0}^{t}\!\left( \langle\partial_t\Pb,\Pbt\rangle + \langle \Pb,\partial_t\Pbt\rangle \right)\intd s + \frac{1}{2}\Vert \Pbt(t)\Vert_{\bL^2}^2
\end{align}
for all $t\in[0,T]$. Note that by \eqref{eq:initialvalue}, the weak solution $\Pb$ takes the initial value in $(\DAO)^*$. Since $\Pbt\in L^4(\DA)\cap L^\infty(\bH^1)$ by Remark~\ref{rem:welldef}, it is an admissible test function in \eqref{eq:weak2}. Hence, by using \eqref{eq:weak2} in \eqref{eq:releng0} above as well as by adding and subtracting $(\mathscr{R}_2[\ubt,\Pbt],\Pb)$, we obtain
\begin{equation}\label{eq:relest1}
    \begin{aligned}
	    \mathscr{E}[\Pb|\Pbt](t) =& \frac{1}{2} \Vert \Pb(t)\Vert_{\bL^2}^2  -(\Pb_0,\Pbt(0))
	    +\mu_2 \int_0^t\! (\Delta\Pb,\Delta\Pbt)\intd s + \gamma_2 \int_0^t\! (\nabla\Pb,\nabla\Pbt) \intd s 
	    \\&+\lambda_2\int_0^t\!((\Pb\cdot\nabla)\Pb,\Pbt)\intd s 
	    +\alpha  \int_0^t\! (\vert \Pb \vert^2 \Pb , \Pbt) \intd s +\beta\int_0^t\! (\Pb,\Pbt)\intd s  +\int_0^t\! ((\ub\cdot\nabla)\Pb,\Pbt)\intd s 
	    \\&+ \frac{1}{2}\int_0^t\! ((\Pb\cdot\nabla)\Pbt-(\Pbt\cdot\nabla)\Pb,\ub)\intd s +\mu_2 \int_0^t\! (\Delta\Pbt,\Delta\Pb)\intd s + \gamma_2 \int_0^t\! (\nabla\Pbt,\nabla\Pb) \intd s
	    \\&+\lambda_2\int_0^t\!  ((\Pbt\cdot\nabla)\Pbt,\Pb)\intd s
	    +\alpha  \int_0^t\! (\vert \Pbt \vert^2 \Pbt , \Pb) \intd s +\beta\int_0^t\! (\Pbt,\Pb)\intd s + \int_0^t\! ((\ubt\cdot\nabla)\Pbt,\Pb)\intd s
	    \\&+ \frac{1}{2}\int_0^t\! ((\Pbt\cdot\nabla)\Pb-(\Pb\cdot\nabla)\Pbt,\ubt)\intd s   - \int_0^t\! (\mathscr{R}_2[\ubt,\Pbt],\Pb) \intd s+ \frac{1}{2}\Vert \Pbt(t)\Vert_{\bL^2}^2.
    \end{aligned}
\end{equation}
We observe that we can write the relative dissipation $\mathscr{W}[\Pb|\Pbt]$ as
\begin{align*}
    \int_0^t\!\mathscr{W}[\Pb|\Pbt]\intd s =& \int_0^t\! \left(\mu_2  \Vert \Delta\Pb \Vert_{\bL^2}^2 + \gamma_2  \Vert \nabla\Pb \Vert_{\bL^2}^2 + \alpha   \Vert \Pb \Vert_{\bL^4}^4 +\beta \Vert \Pb \Vert_{\bL^2}^2    \right)  \intd s \\ 
    &+ \int_0^t\! \left(\mu_2  \Vert \Delta\Pbt \Vert_{\bL^2}^2 +  \gamma_2  \Vert \nabla\Pbt \Vert_{\bL^2}^2 + \alpha   \Vert \Pbt \Vert_{\bL^4}^4  + \beta \Vert \Pbt \Vert_{\bL^2}^2    \right)  \intd s\\
    &-\alpha \int_0^t\!\left( 4(\vert \Pb \vert^2 \Pb, \Pbt) -6 (\vert \Pb \vert^2, \vert \Pbt \vert^2) +4 ( \Pb, \vert \Pbt \vert^2\Pbt)  \right) \intd s \\ 
    &-2\int_0^t\!\left( \mu_2 (\Delta \Pb,\Delta \Pbt) + \gamma_2 (\nabla \Pb, \nabla \Pbt) + \beta (\Pb,\Pbt)\right) \intd s \\
    &+\int_0^t\!\left(\gamma_2^-\Vert \nabla\Pb - \nabla\Pbt \Vert_{\bL^2}^2 +\beta^-\Vert\Pb - \Pbt \Vert_{\bL^2}^2 \right)\intd s.
\end{align*}
Employing the energy inequality \eqref{eq:energy} and the shifted energy equality \eqref{eq:shiftenerg}, we thus come up with
\begin{equation}\label{eq:relest2}
    \begin{aligned}
        \int_0^t\!\mathscr{W}[\Pb|\Pbt]\intd s \le& \frac{1}{2}\left( \Vert \Pb_0\Vert_{\bL^2}^2 -\Vert \Pb(t)\Vert_{\bL^2}^2   +  \Vert \Pbt(0)\Vert_{\bL^2}^2 - \Vert \Pbt(t)\Vert_{\bL^2}^2 \right) + \int_0^t\! (\mathscr{R}_2[\ubt,\Pbt],\Pbt) \intd s\\&
        -\alpha \int_0^t\!\left( 4(\vert \Pb \vert^2 \Pb, \Pbt) -6(\vert \Pb \vert^2, \vert \Pbt \vert^2) +4 ( \Pb, \vert \Pbt \vert^2\Pbt)  \right) \intd s \\&
        -2\int_0^t\!\left( \beta (\Pb,\Pbt) + \gamma_2 (\nabla \Pb, \nabla \Pbt) + \mu_2 (\Delta \Pb,\Delta \Pbt) \right) \intd s \\
        &+\int_0^t\!\left(\gamma_2^-\Vert \nabla\Pb - \nabla\Pbt \Vert_{\bL^2}^2 +\beta^-\Vert\Pb - \Pbt \Vert_{\bL^2}^2 \right)\intd s.
    \end{aligned}
\end{equation}
Here, we have employed the binomial formula
\begin{equation*}
    \vert \boldsymbol{a}-\boldsymbol{b}\vert^4 = \vert \boldsymbol{a} \vert^4 - 4 \vert \boldsymbol{a}\vert^2\boldsymbol{a}\cdot \boldsymbol{b} + 6 \vert \boldsymbol{a} \vert^2 \vert \boldsymbol{b} \vert^2 - 4 \boldsymbol{a} \cdot \boldsymbol{b}\vert \boldsymbol{b} \vert^2 + \vert \boldsymbol{b} \vert^4   
\end{equation*}
for vectors $\boldsymbol{a},\boldsymbol{b}\in\mathbb{R}^3$. Together with the identity
\begin{equation*}
    \begin{aligned}
        &-4(\vert \Pb \vert^2 \Pb, \Pbt) +6 (\vert \Pb \vert^2, \vert \Pbt \vert^2) -4 ( \Pb, \vert \Pbt \vert^2\Pbt)\\
        =& -(\vert \Pb \vert^2 \Pb, \Pbt) -  (\vert \Pbt \vert^2 \Pbt, \Pb) + 3  ((\vert \Pb\vert^2 - \vert \Pbt \vert^2)(\Pbt-\Pb), \Pbt) + 3(\vert \Pbt \vert^2 (\Pbt-\Pb),\Pbt-\Pb),
    \end{aligned}
\end{equation*}
adding \eqref{eq:relest1} and \eqref{eq:relest2} results in
\begin{equation}\label{eq:releng1}
\begin{aligned}
	&\mathscr{E}[\Pb|\Pbt](t)+\int_{0}^{t}\!\mathscr{W}[\Pb|\Pbt]\,\text{d}s
	\\ \le& \mathscr{E}[\Pb_0|\Pbt(0)]+ \int_0^t\! (\mathscr{R}_2[\ubt,\Pbt],\Pbt-\Pb) \intd s+\int_{0}^{t}\!((\ub\cdot\nabla)\Pb,\Pbt)+((\ubt\cdot\nabla)\Pbt,\Pb)\intd s \\
	&+\frac{1}{2}\int_{0}^{t}\!\left(((\Pb\cdot\nabla)\Pbt-(\Pbt\cdot\nabla)\Pb,\ub)+((\Pbt\cdot\nabla)\Pb-(\Pb\cdot\nabla)\Pbt,\ubt)\right)\intd s \\
	&+\lambda_2\int_{0}^{t}\!\left(((\Pb\cdot\nabla)\Pb,\Pbt)+((\Pbt\cdot\nabla)\Pbt,\Pb)\right)\intd s\\ &+\int_0^t\!\left(\gamma_2^-\Vert \nabla\Pb - \nabla\Pbt \Vert_{\bL^2}^2 +\beta^-\Vert\Pb - \Pbt \Vert_{\bL^2}^2 \right)\intd s\\
	&  +3\alpha \int_{0}^{t}\!\left( ((\vert \Pb\vert^2 - \vert \Pbt \vert^2)(\Pb-\Pbt), \Pbt) + (\vert \Pbt \vert^2 (\Pbt-\Pb),\Pbt-\Pb)\right)\intd s\\
	\eqqcolon& \mathscr{E}[\Pb_0|\Pbt(0)] + \int_0^t\! (\mathscr{R}_2[\ubt,\Pbt],\Pbt-\Pb) \intd s + I_1 + \frac{1}{2} I_2 + \lambda_2 I_3 + I_4 + 3\alpha I_5.
\end{aligned}
\end{equation}
It remains to estimate the integral expressions $I_1,\dots,I_5$ against the integral over terms of the relative energy $\mathscr{E}[\Pb,\Pbt]$ and the relative dissipation $\mathscr{W}[\Pb,\Pbt]$.

Since $\ub$ and $\ubt$ are divergence-free, we observe that $(((\ub-\ubt) \cdot \nabla)\Pbt,\Pbt)=0$ and we obtain
\begin{equation*}
    \begin{aligned}
        \vert I_1 \vert =& \left\vert \int_{0}^{t}\!\left(((\ub\cdot\nabla)\Pb,\Pbt)+((\ubt\cdot\nabla)\Pbt,\Pb) + (((\ub-\ubt) \cdot \nabla)\Pbt,\Pbt) \right)\intd s \right\vert\\
        =&\left\vert \int_{0}^{t}\!(((\ub-\ubt) \cdot \nabla)\Pbt,\Pbt - \Pb) \intd s \right\vert.
    \end{aligned}
\end{equation*}
Applying \textsc{H\"older}'s and \textsc{Young}'s inequality as well as Lemma~\ref{lem:coerc} then yields
\begin{equation}
    \begin{aligned}\label{eq:int1}
        \vert I_1 \vert \le& \frac{\delta}{2\varepsilon C(1+\Vert \Pbt \Vert_{L^\infty(\bL^{\frac{6}{5}})}^2)} \int_0^t\!\Vert \ub - \ubt \Vert_{\bL^2}^2 \intd s\\ &+ \int_0^t\!  c_\delta \varepsilon C(1+\Vert \Pbt \Vert_{L^\infty(\bL^\frac{6}{5})}^2) \Vert \nabla \Pbt \Vert_{\bL^\infty}^2\Vert \Pb - \Pbt \Vert_{\bL^2}^2 \intd s \\
        \le&  \delta \int_0^t\! \mathscr{W}[\Pb,\Pbt] \intd s + \int_0^t\!\mathscr{K}[\Pbt]\mathscr{E}[\Pb,\Pbt] \intd s + \frac{1}{2} \int_0^t\! (\mathscr{R}_1[\ubt,\Pbt], A^{-1}(\ubt-\ub)) \intd s.
    \end{aligned}
\end{equation}
Here $0<\delta\le 1$ is fixed in such a way that the factor in front of the terms on the right-hand side of \eqref{eq:relengery} that depend on $\mathscr{W}$ will add up to $1/2$ so that they can be absorbed into the left-hand side. Hence, the constant $c_\delta$ that stems from \textsc{Young}'s inequality can be incorporated into the definition of $\mathscr{K}$. 

An integration-by-parts yields
\begin{equation*}
    \begin{aligned}
        \Vert \nabla\Pb - \nabla\Pbt \Vert_{\bL^2} \le \Vert \Pb - \Pbt \Vert_{\bL^2}^{\frac{1}{2}}\Vert \Delta\Pb - \Delta\Pbt \Vert_{\bL^2}^{\frac{1}{2}}
    \end{aligned}
\end{equation*}
and hence, by \textsc{H\"older}'s inequality, we get
\begin{equation*}
    \begin{aligned}
    \frac{1}{2} \vert I_2 \vert &= \frac{1}{2} \left\vert \int_0^t\!\left( (((\Pb - \Pbt)\cdot\nabla)\Pbt,\ub-\ubt) + ((\Pbt\cdot\nabla)(\Pbt - \Pb),\ub-\ubt) \right)\intd s \right\vert \\
    &\le \frac{1}{2}\int_0^t\!  \Vert \Pb - \Pbt \Vert_{\bL^2}\Vert \nabla\Pbt \Vert_{\bL^\infty} \Vert \ub - \ubt \Vert_{\bL^2} \intd s \\&\qquad\qquad+ \frac{1}{2}\int_0^t\!\Vert \Pbt \Vert_{\bL^\infty }\Vert \Pb - \Pbt \Vert_{\bL^2}^{\frac{1}{2}}\Vert \Delta\Pb - \Delta\Pbt \Vert_{\bL^2}^{\frac{1}{2}} \Vert \ub - \ubt \Vert_{\bL^2}   \intd s.
    \end{aligned}
\end{equation*}
As in \eqref{eq:int1} we can then use \textsc{Young}'s inequality and Lemma~\ref{lem:coerc} to find
\begin{equation}\label{eq:int2}
    \begin{aligned}
         \frac{1}{2} \vert I_2 \vert \le \delta \int_0^t\! \mathscr{W}[\Pb|\Pbt] \intd s + \int_0^t\!\mathscr{K}[\Pbt]\mathscr{E}[\Pb|\Pbt] \intd s + \frac{1}{2}\int_0^t\! (\mathscr{R}_1[\ubt,\Pbt], A^{-1}(\ubt-\ub)) \intd s.
    \end{aligned}
\end{equation}
Since also $\Pb-\Pbt$ is divergence-free, we find, similarly as for $I_1$, that
\begin{equation}\label{eq:int3}
    \begin{aligned}
        \lambda_2\vert  I_3 \vert =&\lambda_2 \left\vert  \int_{0}^{t}\!(((\Pb-\Pbt) \cdot \nabla)\Pbt,\Pbt - \Pb) \intd s \right\vert\\
        \le& \lambda_2 \int_0^t\! \Vert \Pb - \Pbt \Vert_{\bL^2}  \Vert \nabla\Pbt \Vert_{\bL^\infty} \Vert \Pb - \Pbt \Vert_{\bL^2} \intd s\\
        \le&   \int_0^t\!\mathscr{K}[\Pbt]\mathscr{E}[\Pb|\Pbt] \intd s + \delta \int_0^t\! \mathscr{W}[\Pb|\Pbt] \intd s.
    \end{aligned}
\end{equation}
Performing an integration-by-parts and using \textsc{H\"older}'s as well as \textsc{Young}'s inequality gives
\begin{equation}\label{eq:int4}
    \begin{aligned}
        \vert  I_4 \vert \le&  \int_{0}^{t}\!\left( \gamma_2^-\Vert\Pb-\Pbt\Vert_{\bL^2} \Vert\Delta \Pb - \Delta \Pbt\Vert_{\bL^2}+ \beta^-\Vert\Pb-\Pbt\Vert_{\bL^2}^2\right)\intd s\\
        \le&  \int_{0}^{t}\!c_\delta \Vert\Pb-\Pbt\Vert_{\bL^2}^2 \intd s + \delta \int_0^t \mu_2 \Vert\Delta \Pb - \Delta \Pbt\Vert_{\bL^2}^2 \intd s\\
        \le&  \int_0^t\!\mathscr{K}[\Pbt]\mathscr{E}[\Pb|\Pbt] \intd s + \delta \int_0^t\! \mathscr{W}[\Pb|\Pbt] \intd s.
    \end{aligned}
\end{equation}
The estimate
\begin{equation*}
    \begin{aligned}
    \vert \Pb \vert^2 - \vert \Pbt \vert^2 =& 2(\vert \Pb \vert - \vert \Pbt \vert)\vert \Pbt \vert + (\vert \Pbt \vert - \vert \Pb \vert)^2 \\
    &\le 2\vert \Pb  -  \Pbt \vert\vert \Pbt \vert + \vert \Pbt  -  \Pb \vert^2
    \end{aligned}
\end{equation*}
together with \textsc{H\"older}'s inequality then gives
\begin{equation*}
    \begin{aligned}
         3\alpha \vert  I_5 \vert \le&   3\alpha \int_{0}^{t}\! \left( 3 \Vert \Pb - \Pbt \Vert_{\bL^2}^2 \Vert  \Pbt \Vert_{\bL^\infty}^2  + \Vert \Pb - \Pbt \Vert_{\bL^4}^3 \Vert \Pbt \Vert_{\bL^4} \right) \intd s.
    \end{aligned}
\end{equation*}
By a well-known interpolation inequality for $\bL^p(\Omega)$ spaces and the \textsc{Sobolev} embedding $ \bH^2(\Omega)\hookrightarrow\bL^\infty(\Omega)$, we can further estimate 
\begin{equation*}
    \begin{aligned}
         \Vert \Pb - \Pbt \Vert_{\bL^4}^3 \Vert \Pbt \Vert_{\bL^4} \le c\Vert \Pb - \Pbt \Vert_{\bL^4}^2 \Vert \Pb - \Pbt \Vert_{\bL^2}^{\frac{1}{2}} \Vert \Delta \Pb - \Delta \Pbt \Vert_{\bL^2}^{\frac{1}{2}} \Vert \Pbt \Vert_{\bL^4}.
    \end{aligned}
\end{equation*}
With \textsc{Young}'s inequality, we see that
\begin{equation}\label{eq:int5}
    \begin{aligned}
         3\alpha \vert  I_5 \vert \le&   \int_0^t\!\mathscr{K}[\Pbt]\mathscr{E}[\Pb|\Pbt] \intd s + \delta \int_0^t\! \mathscr{W}[\Pb|\Pbt] \intd s. 
    \end{aligned}
\end{equation}
Applying \eqref{eq:int1}--\eqref{eq:int5} to \eqref{eq:releng1} with an appropriate choice of $\delta$ finally yields
\begin{align*}
    \mathscr{E}[\Pb|\Pbt](t)+\int_{0}^{t}\!\mathscr{W}[\Pb|\Pbt]\intd s \le \mathscr{E}[\Pb_0|\Pbt(0)] +& \frac{1}{2}\int_{0}^{t}\!\mathscr{W}[\Pb|\Pbt]\intd s+ \int_{0}^{t}\!\mathscr{K}[\Pbt]\mathscr{E}[\Pb|\Pbt]\intd s\\&+ \int_{0}^{t}\!\left(\mathscr{R}[\ubt,\Pbt],\left(\begin{array}{c} A^{-1}(\ubt-\ub)\\ \Pbt-\Pb \end{array}\right)\right)\intd s,
\end{align*}
which is the assertion.

\end{proof}
With the relative energy inequality available, an immediate consequence is then the weak-strong uniqueness of the weak solutions constructed in Section \ref{sec:existence}, which by Lemma \ref{lem:energyineq} fulfill the necessary energy inequality. First, we make our understanding of strong solutions precise.

\begin{definition}[Strong solution]\label{def:strong}
Let $\Pbt_0\in\DAO$. A pair $(\ubt,\Pbt)\in L^2(\DA)\times W^{1,2}(\bH^4\cap\bH^1_{0,\sigma})$ is called a strong solution to \eqref{eq:model} if the equations \eqref{eq:weak1} and \eqref{eq:weak2} hold and $\Pbt(0)=\Pbt_0$ almost everywhere in $\Omega$.
\end{definition}
It is then clear that strong solutions lie in the class of possible test functions for the relative energy inequality from which the weak-strong uniqueness follows.

\begin{corollary}[Weak-strong uniqueness]
Let $(\ub,\Pb)$ be a weak solution to \eqref{eq:model} in the sense of Definition~\ref{def:weak} that additionally fulfills the energy inequality \eqref{eq:energy} and let $(\ubt,\Pbt)$ be a strong solution to \eqref{eq:model} in the sense of Definition~\ref{def:strong}, starting from the same initial datum $\Pbt_0\in\bH^4(\Omega)\cap\bH^1_{0,\sigma}(\Omega)$. Then $(\ub,\Pb)$ is unique and coincides with $(\ubt,\Pbt)$. 
\end{corollary}
\begin{proof}
Since for any strong solution $(\ubt,\Pbt)$ the expression $\mathscr{R}[\ubt,\Pbt]$ vanishes, dropping the (positive) term $\frac{1}{2}\int_0^t\!\mathscr{W}[\Pb|\Pbt]\intd t$ on the left-hand side of \eqref{eq:relengery}, and an application of \textsc{Gronwall}'s lemma yields the result.  
\end{proof}

\section{Relation to a phenomenological model}\label{sec:connect}

The numerical experiments carried out in \textsc{Reinken} et al.~\cite{reinken} suggest that, for a very small parameter $\varepsilon$, the dynamics of the active fluid is dominated by the movement of the microswimmers, and the influence of the velocity field of the suspension fluid can be neglected. Hence, the behaviour of the fluid as a whole is described by only the vector field $\Pb$. We see this reflected in the governing equations via the following formal calculations. Setting $\varepsilon=0$ and thus $\mu_1=\gamma_1=\lambda_1=0$ in \eqref{classical} leaves us with the decoupled system 
\begin{align*}
	-\Delta\ub+\nabla \pi_1&=0,\\
	\partial_t\Pb + \mu_2\Delta^2\Pb - \gamma_2 \Delta\Pb + \lambda_2 (\Pb\cdot\nabla)\Pb +\alpha \vert \Pb \vert^2\Pb + \beta \Pb\quad& \\+ (\ub\cdot\nabla) \Pb +\kappa \usym\Pb - \uskw\Pb + \nabla\pi_2 &= 0, \\
	\nabla\cdot\ub  = \nabla\cdot\Pb &= 0
\end{align*}
together with \eqref{eq:initcond}. Inserting the trivial solution $\ub=0$ and $\pi_1$ constant reduces this to
\begin{equation}\label{eq:simple}
	\begin{aligned}
	    \partial_t\Pb + \mu_2\Delta^2\Pb - \gamma_2 \Delta\Pb + \lambda_2 (\Pb\cdot\nabla)\Pb +\alpha \vert \Pb \vert^2\Pb + \beta \Pb + \nabla\pi_2 &= 0
	\end{aligned}
\end{equation}
with \eqref{eq:initcond}. The equations above coincide with the ones in a phenomenological model that was proposed in \textsc{Wensink} et al.~\cite{wensink}. The existence of strong solutions to a generalisation of these equations on the whole space was shown in \textsc{Zanger} et al.~\cite{zanger} and can be adapted to the case of a bounded domain. The goal of this section is to make this relation rigorous by showing that weak solutions to \eqref{eq:model} constructed in Section \ref{sec:existence} converge to strong solutions to \eqref{eq:simple} with \eqref{eq:initcond} as $\varepsilon\rightarrow 0$ by employing the relative energy inequality \eqref{eq:relengery} from Theorem~\ref{thm:releng}.  

\begin{theorem}
Let $(\Pbez)_{\varepsilon>0}\subset\DAO$ be a family of initial values and for $\varepsilon>0$ let $(\ube,\Pbe)$ be a weak solution to \eqref{eq:model}, in the sense of Definition~\ref{def:weak}, starting from the initial datum $\Pbez\in\DAO$ that fulfills the energy inequality \eqref{eq:energy}. Furthermore, let $\Pbt\in W^{1,2}(\bH^4\cap\bH^1_{0,\sigma})$ be a strong solution to \eqref{eq:simple} together with \eqref{eq:initcond} starting from the initial datum $\Pbt_0\in\bH^4(\Omega)\cap\bH^1_{0,\sigma}(\Omega)$. If $\Pbez\rightarrow\Pbt_0$ in $\bLs^2(\Omega)$ then for the corresponding family of weak solutions $(\ube,\Pbe)_{\varepsilon>0}$ there holds
\begin{align*}
(\ube,\Pbe)\rightarrow (0,\Pbt) \text{ in } L^2(\bLs^2) \times L^\infty(\bL^2)\cap L^4(\bL^4) \cap L^2(\DA)
\end{align*}
as $\varepsilon\rightarrow 0$.
\end{theorem}

\begin{proof}
First, we denote by $\mathscr{K}_\varepsilon[\Pbt]=\mathscr{K}[\Pbt]$ the expression defined in \eqref{eq:defk}, depending on $\varepsilon$ via $(\ube,\Pbe)$. The pair $(0,\Pbt)$ then is admissible to test with in the relative energy inequality \eqref{eq:relengery}. An application of \textsc{Gronwall}'s lemma yields
\begin{equation}\label{eq:limit1}
    \begin{aligned}
        &\mathscr{E}[\Pbe|\Pbt](t) + \frac{1}{2}\int_0^t\!\mathscr{W}[\Pbe|\Pbt] \exp\left(\int_s^t\!\mathscr{K}_\varepsilon[\Pbt]\mathrm{d}\tau\right) \intd s \\
        \le& \mathscr{E}[\Pbez|\Pbt_0]\exp\left( \int_0^t\!\mathscr{K}_\varepsilon[\Pbt]\mathrm{d}s\right) + \int_0^t\!\left(\mathscr{R}[0,\Pbt],\left(\begin{array}{c}  A^{-1}(-\ube)\\ \Pbt-\Pbe  \end{array}\right)\right) \exp\left( \int_s^t\!\mathscr{K}_\varepsilon[\Pbt]\mathrm{d}\tau\right)\mathrm{d}s
    \end{aligned}
\end{equation}
for almost all $t\in(0,T)$. Since $\mathscr{K}_\varepsilon[\Pbt]\ge 0$, it suffices to show that the right-hand side of this inequality vanishes to obtain the desired convergences. First, we note that $\int_s^t\!\mathscr{K}_\varepsilon[\Pbt]\mathrm{d}\tau$ is uniformly  bounded in $\varepsilon$ (in fact vanishes as $\varepsilon\rightarrow 0$) since $\varepsilon$ only appears as the leading constant in $\mathscr{K}_\varepsilon[\Pbt]$. Hence, with the convergence $\Pbez\rightarrow\Pbt_0$ in $\bL^2(\Omega)$, the first term on the right-hand side of \eqref{eq:limit1} vanishes as $\varepsilon\rightarrow 0$. Furthermore, since $\Pbt$ is a strong solution to \eqref{eq:simple} and therefore $\mathscr{R}_2[0,\Pbt]$ vanishes, it only remains to show that
\begin{align*}
    \int_0^t\!(\mathscr{R}_1[0,\Pbt], A^{-1}\ube)\intd s =& \varepsilon\int_0^t\! (\tilde{\mu_2}\Delta^2\Pbt-\tilde{\gamma_1}\Delta\Pbt+\tilde{\lambda_1} (\Pbt\cdot\nabla)\Pbt,A^{-1}\ube)\intd s
\end{align*}
tends to zero as $\varepsilon\rightarrow 0$. Using the continuity of $A^{-1} \colon\bLs^2(\Omega)\rightarrow\DAO\hookrightarrow\bLs^2(\Omega)$ and estimate \eqref{eq:energy2}, we find
\begin{align*}
    \Vert A^{-1}\ube \Vert_{L^2(\bL^2)} \le \varepsilon c(\Vert \Pbez \Vert_{\bL^2} + 1).
\end{align*}
Together with the \textsc{Cauchy--Schwarz} inequality and \eqref{eq:coupling}, we finally obtain
\begin{align*}
    \int_0^t\!(\mathscr{R}_1[0,\Pbt], A^{-1}\ube)\intd s \le& \varepsilon  \Vert \tilde{\mu_2}\Delta^2\Pbt-\tilde{\gamma_1}\Delta\Pbt+\tilde{\lambda_1} (\Pbt\cdot\nabla)\Pbt \Vert_{L^2(\bL^2)}\Vert A^{-1}\ube \Vert_{L^2(\bL^2)}\\
    \le& \varepsilon^2 c  \Vert \tilde{\mu_2}\Delta^2\Pbt-\tilde{\gamma_1}\Delta\Pbt+\tilde{\lambda_1} (\Pbt\cdot\nabla)\Pbt \Vert_{L^2(\bL^2)}(\Vert \Pbez \Vert_{\bL^2} + 1),
\end{align*}
which in turn yields the statement.

\end{proof}

\section*{Acknowledgement}

This work has been supported by Deutsche Forschungsgemeinschaft through Collaborative Research Center 910 “Control of self-organizing nonlinear systems:~Theoretical methods and concepts of application”.

\bibliographystyle{mybibstyle}
\bibliography{references}

\end{document}